\documentclass{CSML}

\newcommand{\dOi}{13(1:18)2017}
\lmcsheading{}{1--11}{}{}{Mar.~16, 2016}{Mar.~30, 2017}{}

\DeclareMathOperator{\cl}{cl}
\DeclareMathOperator{\dom}{dom}
\DeclareMathOperator{\fat}{fat}

\DeclareMathOperator{\inte}{int}
\DeclareMathOperator{\len}{len}
\DeclareMathOperator{\upper}{\uparrow}

\newcommand{\widebigcap}[1]{%
\ \bigcap_{\mathclap{\substack{#1}}}\ }
\newcommand{\widebigcup}[1]{%
\ \bigcup_{\mathclap{\substack{#1}}}\ }

\usepackage{mathtools}
\usepackage{hyperref}
\hypersetup{
  colorlinks=true,
  linkcolor=black,
  citecolor=black,
  filecolor=black,
  urlcolor=black,
}

\newcommand{\bd}{\partial}

\begin{document}
\title{Existence of strongly proper dyadic subbases}
\author{Yasuyuki Tsukamoto}
\address{Graduate School of Human and Environmental Studies,
Kyoto University, Japan.}
\email{tsukamoto@i.h.kyoto-u.ac.jp}


\begin{abstract}
\noindent
We consider a topological space with its subbase
 which induces a coding for each point.
Every second-countable Hausdorff space has a subbase
 that is the union of countably many pairs of disjoint open subsets.
A dyadic subbase is such a subbase with a fixed enumeration.
If a dyadic subbase is given,
 then we obtain a domain representation of the given space.
The properness and the strong properness of dyadic subbases have been studied,
 and it is known that
 every strongly proper dyadic subbase induces an admissible domain representation
 regardless of its enumeration.
We show that every locally compact separable metric space
 has a strongly proper dyadic subbase.
\end{abstract}
\maketitle
\section{Introduction}
Let $X$ be a second-countable Hausdorff space.
We consider a subbase of $X$
 which induces a coding for each point of $X$.
Since $X$ is second-countable,
 it has a subbase that is the union of countably many pairs
 of disjoint open subsets.
A \emph{dyadic subbase} $S$ is such a subbase with a fixed enumeration.
Let $(S_n^0,S_n^1)$, $n<\omega$, be the $n$-th pair of $S$.
For each point $x\in X$,
 the $n$-th digit of the coding of $x$ will be decided between $0$ and $1$
 depending on which of $S_n^0$ and $S_n^1$ contains $x$.
Since $x$ may belong to none of the two, we allow undefinedness in the coding,
 and the \emph{bottom character} $\bot$ is used
 in the sequence unless it is decided.
Every sequence that contains the bottom character
 is called a \emph{bottomed sequence.}

When a dyadic subbase $S$ is given,
 each point $x$ of $X$ is represented
 by a bottomed sequence $\varphi_S(x)\in \{0,1,\bot\}^{\omega}$.
Every finite prefix of the bottomed sequence $\varphi_S(x)$
 can be considered as a finite time state of the output of computation,
 and the properties of the set
 $K_S:=\{\varphi_S(x)|_n: x\in X,\,n<\omega\}$
 of these sequences have been studied.
If $S$ is proper and $K_S$ is a conditional upper semilattice with least element (cusl),
 then we obtain an admissible domain representation of $X$ \cite{tsuiki15}.
Whether $K_S$ is a cusl depends not only on $S$
 itself but also on the enumeration of $S$.
It has been proved that
 $K_S$ is a cusl regardless of the enumeration of $S$
 if and only if $S$ is strongly proper
 \cite{tsukamoto-to-appear}.
The definitions of proper and strongly proper dyadic subbases
 will be recalled in the next section.

We study a condition which ensures the existence of strongly proper dyadic subbases.
It has been proved that every second-countable regular Hausdorff space
 has a proper dyadic subbase \cite{ohta13,ohta11}.
First we give another proof of this fact that uses the metric
 induced by Urysohn's metrisation theorem.
Then we show that every locally compact separable metric space
 has a strongly proper dyadic subbase.

\section{Preliminaries}
\subsection{Bottomed sequences}
Let $\Sigma$ be a finite set containing the bottom character $\bot$,
 and $\omega$ the first infinite ordinal.
As usual, an element of $\omega$ is identified with
 the set of its predecessors.

For an ordinal number $n\leq \omega$,
 $\Sigma^n$ denotes the set of sequences of
 elements of $\Sigma$ of length $n$,
 i.e., the maps from $n$ to $\Sigma$. 
We identify a sequence $\sigma \in \Sigma^n$
 with its infinite extension
\begin{equation*}
\sigma(k):=\left\{
 \begin{array}{ll}
  \sigma(k)& (k< n)\\
  \bot & (n\leq k<\omega).
 \end{array}
\right.
\end{equation*}
By this identification, we have $\Sigma^n\subseteq \Sigma^m$ if $n\leq m$.
For a sequence $\sigma\in \Sigma^{\omega}$,
 its domain is defined as $\dom(\sigma):=\{k: \sigma(k)\neq \bot\}$
 and its length as $\len(\sigma):=\min\{n\leq\omega: \dom(\sigma) \subseteq n\}$.
A sequence $\sigma\in \Sigma^{\omega}$ is \emph{compact}
 if $\len(\sigma)$ is finite.
The set of compact sequences is denoted by $\Sigma^*:=\bigcup_{n<\omega}\Sigma^n$.

Let $\sigma$ and $\tau$ be sequences of $\Sigma$,
 $a\in \Sigma$ an element  and $n<\omega$ a finite ordinal.
$\sigma[n\mapsto a]$ is the sequence
which maps $n$ to $a$ and equals $\sigma$ elsewhere.
If $\sigma$ is compact,
 then a concatenation $\sigma\tau\in \Sigma^{\omega}$ is defined as
\begin{equation*}
\sigma\tau (k):=\left\{
 \begin{array}{ll}
  \sigma(k)&(k< \len(\sigma))\\
  \tau(k-\len(\sigma))&(\len(\sigma)\leq k<\omega).
 \end{array}\right.
\end{equation*}
The $n$-fold concatenation of $\sigma$ with itself is denoted by $\sigma^n$.
Elements $a\in \Sigma$ are identified with sequences of length one.

\subsection{The domain $\mathbb{T}^{\omega}$}
We set $\mathbb{T}:=\{0,1,\bot\}$ ordered by $\bot\sqsubseteq0$ and $\bot\sqsubseteq 1$.
A $T_0$ topology of $\mathbb{T}$ is defined as $\{\emptyset, \{0\},\{1\},\{0,1\},\mathbb{T}\}$.
We equip $\mathbb{T}^{\omega}$
 with the product order and the product topology.
For all $\sigma\in \mathbb{T}^{\omega}$,
 we set $\upper\sigma:=\{\tau\in \mathbb{T}^{\omega}:\tau\sqsupseteq \sigma\}$,
 and the family $\{\upper\sigma:\sigma\in \mathbb{T}^*\}$ forms
 a base of $\mathbb{T}^{\omega}$ \cite{plotkin78}.

\subsection{Dyadic subbases}
We review proper and strongly proper dyadic subbases
 \cite{tsuiki04b,tsuiki15,tsukamoto-to-appear}.
In this section,
 $X=(X,\mathfrak{O})$ is a second-countable Hausdorff space.
For a subset $A$ of $X$,
 $\cl A$ denotes the closure of $A$, $\inte A$ the interior of $A$.
The exterior of a subset $A\subseteq X$ is
 the set $\inte (X\setminus A)$ or equivalently $X\setminus \cl A$.

\begin{defi}\label{def:dyadicsubbase}
A \emph{dyadic subbase} of $X$ is
 a map $S:\omega\times \{0,1\}\rightarrow \mathfrak{O}$
 such that
\begin{enumerate}
\item $\{S(n,a):  n<\omega,\,a\in \{0,1\}\}$ is a subbase of $X$,
\item $S(n,0)\cap S(n,1)=\emptyset$ for all $n<\omega.$
\end{enumerate}
\end{defi}
For readability, $S(n,a)$ is denoted by $S_n^a$.
We set $S_n^{\bd}:=X\setminus(S_n^0\cup S_n^1)$ for $n<\omega$.
If $S_n^0$ and $S_n^1$ are the exteriors of each other,
then $S_n^{\bd}$ is their common boundary,
and we have $\cl S_n^a=S_n^a\cup S_n^{\bd}$ for $a\in \{0,1,\bd \}$.
We use the notations
\begin{align}
S(\sigma)&:=\widebigcap{k\in \dom(\sigma)}S_k^{\sigma(k)}, \label{eq:s-open}
\end{align}
\begin{align}
\bar{S}(\sigma)&:=\widebigcap{k\in \dom(\sigma)}
 (S_k^{\sigma(k)}\cup S_k^{\bd})\label{eq:s-closed}
\end{align}
 for all $\sigma\in \{0,1,\bd,\bot\}^{\omega}$.
\begin{defi}\label{dfn:PandI}
A dyadic subbase $S$ of $X$ is
\begin{enumerate}
\item \emph{proper} if $\bar{S}(\sigma)=\cl S(\sigma)
\text{ for all } \sigma\in \mathbb{T}^*$.
\item \emph{strongly proper} if
 $\bar{S}(\sigma)=\cl S(\sigma)$  for all $\sigma\in\{0,1,\bd,\bot\}^*$.
\end{enumerate}
\end{defi}

For all $\sigma\in \mathbb{T}^*$, by De Morgan's law, we obtain 
\begin{equation*}
\bar{S}(\sigma)=\widebigcap{k\in \dom(\sigma)} (S_k^{\sigma(k)}\cup S_k^{\bd})
=\widebigcap{k\in \dom(\sigma)} X\setminus S_k^{1-\sigma(k)}
=X\setminus\widebigcup{k\in \dom(\sigma)} S_k^{1-\sigma(k)},
\end{equation*}
and therefore, we have
 $X\setminus\bar{S}(\sigma)=\bigcup_{k\in \dom(\sigma)} S_k^{1-\sigma(k)}$.
We have the following characterization of proper dyadic subbases.
\begin{prop}\label{prop:ch-proper}
A dyadic subbase $S$ of $X$ is
 proper if and only if $S(\sigma)$ and $X\setminus \bar{S}(\sigma)$ 
 are the exteriors of each other for all $\sigma\in \mathbb{T}^*$.
\end{prop}
\begin{proof}
For all $\sigma\in \mathbb{T}^*$, we have $\bar{S}(\sigma) =\cl S(\sigma)$
 if and only if $X\setminus \bar{S}(\sigma)=X\setminus\cl S(\sigma)$
 if and only if $X\setminus \bar{S}(\sigma)$ is the exterior of  $S(\sigma)$.

We show that if $S$ is proper,
 then for all $\sigma\in \mathbb{T}^*$, $S(\sigma)$ is the exterior of $X\setminus \bar{S}(\sigma)$,
i.e., $S(\sigma)$ is the interior of $\bar{S}(\sigma)$.
Suppose that $S$ is proper.
For all $n<\omega$ and $a\in \{0,1\}$,
 we have $\cl S_n^{1-a}=S_n^{1-a}\cup S_n^{\bd}=X\setminus S_n^a$,
 and therefore, 
\begin{equation*}
S_n^a=X\setminus \cl S_n^{1-a} = \inte (X\setminus S_n^{1-a})=
\inte (S_n^a \cup S_n^{\bd}).
\end{equation*}
Hence, for all $\sigma\in \mathbb{T}^*$, we obtain
\begin{equation}
\inte\bar{S}(\sigma)
= \inte \widebigcap{k\in \dom(\sigma)} (S_k^{\sigma(k)}\cup S_k^{\bd})
= \widebigcap{k\in \dom(\sigma)} \inte (S_k^{\sigma(k)}\cup S_k^{\bd})
= \widebigcap{k\in \dom(\sigma)} S_k^{\sigma(k)}
= S(\sigma).\tag*{\qEd}
\end{equation}
\def\popQED{}
\end{proof}

Every sequence $\sigma\in \{0,1,\bd,\bot\}^*$ is decomposed into two sequences
\begin{align}
\sigma_{0,1}(n)&:=\left\{\begin{array}{ll}
\sigma(n)&(\text{if }\sigma(n)\in \{0,1\})\\ 
\bot &(\text{otherwise})\end{array}\right.
\quad \text{and}\\
\sigma_{\bd}(n)&:=\left\{\begin{array}{ll}
\sigma(n)&(\text{if }\sigma(n)=\bd)\\ 
\bot &(\text{otherwise}).\end{array}\right.
\end{align}
We have $S(\sigma)=S(\sigma_{0,1})\cap S(\sigma_{\bd})$,
 and $S(\sigma)$ is an open subset of $S(\sigma_{\bd})$.
Since $S(\sigma_{\bd})$ is closed,
 we have $\bar{S}(\sigma)=\bar{S}(\sigma_{0,1})\cap S(\sigma_{\bd})$.
Similarly to Proposition \ref{prop:ch-proper},
 if $S$ is strongly proper,
 then $S(\sigma)$ and $S(\sigma_{\bd})\setminus \bar{S}(\sigma)$
 are the exteriors of each other in the space $S(\sigma_{\bd})$.
\begin{prop}
A dyadic subbase $S$ of $X$ is
strongly proper if and only if  for all $\sigma\in \{0,1,\bd,\bot\}^*$,
 $S(\sigma)$ and $S(\sigma_{\bd})\setminus \bar{S}(\sigma)$
 are the exteriors of each other in the space $S(\sigma_{\bd})$.
\end{prop}
\begin{proof}
Since $S(\sigma_{\bd})$ is closed,
 $\cl (S(\sigma_{0,1})\cap S(\sigma_{\bd}))$ equals the closure of 
 $S(\sigma_{0,1})\cap S(\sigma_{\bd})$ in the space $S(\sigma_{\bd})$.
Therefore, we have $\cl S(\sigma)=\bar{S}(\sigma)$
if and only if
 $S(\sigma_{\bd})\setminus \bar{S}(\sigma)= S(\sigma_{\bd})\setminus \cl S(\sigma)$
if and only if $S(\sigma_{\bd})\setminus \bar{S}(\sigma)$ is the 
exterior of $S(\sigma)$ in the space $S(\sigma_{\bd})$.
Similarly to the proof of Proposition \ref{prop:ch-proper}, we can see that
 if $S$ is strongly proper, then $S(\sigma)$ is the interior of $\bar{S}(\sigma)$
 in the space $S(\sigma_{\bd})$.
\end{proof}

Let $S$ be a dyadic subbase of $X$.
We define a map
 $\varphi_S: X \rightarrow \mathbb{T}^{\omega}$ as
\begin{equation}
\varphi_S(x)(n):=\left\{\begin{array}{ll}
0 & (x\in S_n^0)\\ 1 & (x\in S_n^1)\\ \bot & (\text{otherwise})
\end{array}\right. 
\end{equation}
 for $x\in X$ and $n<\omega$.
We set
\begin{align*}
K_S&:=\{ \varphi_S(x)|_n : x\in X,\, n<\omega \},\\
D_S&:=\{\sigma \in \mathbb{T}^{\omega}:
 (\forall n<\omega)(\sigma|_n \in K_S) \}.
\end{align*}
The set $D_S$ is an algebraic pointed dcpo that is the ideal completion of $K_S$.
We quote some results known from \cite{tsuiki15, tsukamoto-to-appear}:
If $S$ is a proper dyadic subbase of a regular Hausdorff space $X$,
 then $X$ is embedded in the space of minimal elements of $D_S\setminus K_S$.
Moreover, if $X$ is compact in addition, then 
 we have a quotient map $\rho_S:D_S\setminus K_S \rightarrow X$
 and $X$ is homeomorphic to the space of minimal elements of $D_S\setminus K_S$.
The triple $(D_S, D_S\setminus K_S,\rho_S)$ is a domain representation of $X$.
If $K_S$ is a conditional upper semilattice with least element (cusl),
 then the domain representation is admissible (\cite{tsuiki15}).
Whether $K_S$ is a cusl depends not only on
 the subbase $\{S_n^a:n<\omega,a\in \{0,1\}\}$ itself
 but also on the enumeration of $S$.
It has been proved that  $K_S$ is a cusl regardless of the enumeration of $S$
 if and only if $S$ is strongly proper (\cite{tsukamoto-to-appear}).
We refer the reader to \cite{blanck00}
for the notion of domain representations of topological spaces.


We give an example of a strongly proper dyadic subbase
 and show some modifications. 
\begin{exa}[The Gray subbase of the unit interval]
Let $I$ be the unit interval $[0,1]\subseteq \mathbb{R}$.
For every non-negative integer $n$, we define a function $f_n:I\rightarrow\mathbb{R}$ as
$f_n(x):=-\cos(2^n\pi x)$.
We set $G_n^0:=f_n^{-1}((-\infty,0))$ and $G_n^1:=f_n^{-1}((0,\infty))$ for all $n<\omega$.
The family $\{G_n^a:n<\omega,a\in\{0,1\}\}$ forms a subbase of $I$, the {\em Gray subbase.}

We show that the Gray subbase is a strongly proper dyadic subbase.
For every $n<\omega$, we have $f_n^{-1}(0)=\{(2k+1)/2^{n+1}: k=0,1,\dots,2^n-1\}$,
 and the set $G_n^{\bd}=f_n^{-1}(0)$ is the common boundary of $G_n^0$ and $G_n^1$.
Note that $G_m^{\bd}$ and $G_n^{\bd}$ are disjoint if $m\neq n$.
Let $\sigma\in\{0,1,\bd,\bot\}^*$ be a bottomed sequence.
We have $\bar{G}(\sigma)\supseteq \cl G(\sigma)$
 because $\bar{G}(\sigma)$ is closed and contains $G(\sigma)$.
Suppose that $x$ belongs to $\bar{G}(\sigma)$,
 and we assume $x\in G_m^{\bd}$ for an $m\in\dom(\sigma)$,
 otherwise $x$ belongs to $G(\sigma)$.
Since $G_m^{\bd}$ and $G_k^{\bd}$ are disjoint
 for all $k\in \dom(\sigma)\setminus\{m\}$,
 we get $x\in \cl G_m^{\sigma(m)}\cap G(\sigma[m\mapsto\bot])$
 and $\sigma[m\mapsto\bot]\in\mathbb{T}^*$.
For every open neighbourhood $V$ of $x$, we have
 $V\cap \cl G_m^{\sigma(m)}\cap G(\sigma[m\mapsto\bot])\neq \emptyset$.
Since $V\cap G(\sigma[m\mapsto\bot])$ is open,
 we get $V\cap G_m^{\sigma(m)}\cap G(\sigma[m\mapsto\bot])\neq \emptyset$.
Therefore, every open neighbourhood $V$ of $x$ intersects
 with $G_m^{\sigma(m)}\cap G(\sigma[m\mapsto \bot])=G(\sigma)$ non-trivially.
Hence, we obtain $x\in\cl G(\sigma)$.
\end{exa}

\begin{rem}
We allow duplications in dyadic subbases,
 and a one point set $\{x\}$ has a strongly proper dyadic subbase given by
 $S_n^0=\{x\}$ and $S_n^1=\emptyset$ for all $n<\omega$.
However, such a duplication induces non-properness in general.
Let $G$ be the Gray subbase of the unit interval.
We have $\cl(G_n^0 \cap G_n^1)=\emptyset$
 and $\cl G_n^0\cap \cl G_n^1=G_n^{\bd}\neq\emptyset$.
Therefore, if we enumerate a pair in $G$ twice, then we obtain a non-proper dyadic subbase.
\end{rem}

\begin{exa}
Let $I$ be the unit interval with the Gray subbase $G$.
We define $X$ as the one point compactification of $I\setminus\{1/4,3/4\}$,
 and let $p$ be the added point.
Since we have $\varphi_G(1/4)=0\bot10^{\omega}$ and $\varphi_G(3/4)=1\bot10^{\omega}$, 
$X$ has a dyadic subbase $S$ given by
\begin{equation*}
\varphi_S(x):=\left\{
 \begin{array}{ll}\bot\bot10^{\omega} &(\text{if }x=p)\\
 \varphi_G(x)&(\text{otherwise}).\end{array}
\right.
\end{equation*}
We have $p\in \cl (S_0^a\cap S_1^b)$ for all $a,b\in \{0,1\}$,
 and the properness of $S$ can be proved similarly.
However, $S$ is not strongly proper because we have $p\in S_0^{\bd}\cap \cl S_1^0$
 whereas $p\not\in \cl (S_0^{\bd}\cap S_1^0)$.
\end{exa}

\section{Existence of strongly proper dyadic subbases}
We will show the following.
\begin{thm}\label{thm:locally-compact-st-proper}
Every locally compact separable metric space has a strongly proper dyadic subbase. 
\end{thm}
Every separable metric space is second-countable and regular Hausdorff.
Urysohn's metrization theorem states that 
every second-countable regular Hausdorff space is metrisable.
Therefore, Theorem \ref{thm:locally-compact-st-proper} states that
 every locally compact second-countable Hausdorff space
 has a strongly proper dyadic subbase.
It is still an open problem whether every separable metric space has a strongly proper dyadic subbase.

\subsection{Existence of proper dyadic subbases}
First, we show the following.
\begin{prop}\label{prop:existence-proper}
Every separable metric space $X=(X,d)$ has a proper dyadic subbase.
\end{prop}
Proposition \ref{prop:existence-proper} has been proved already in \cite{ohta13,ohta11}.
Using the metric directly, we give another proof of this fact.
The proof is simpler than the previous proofs,
 and the idea is useful for proving Theorem \ref{thm:locally-compact-st-proper}.
Let $f:X\rightarrow \mathbb{R}$ be a function, $c$ a real number.
We use the notations
\begin{equation}
\begin{split}
U^0(f,c)&:=\{x\in X: f(x)<c\},\\
U^1(f,c)&:=\{x\in X: f(x)>c\},\\
U^{\bd}(f,c)&:=f^{-1}(c).
\end{split}
\end{equation}
We will construct a dyadic subbase
 $S:\omega \times \{0,1\} \rightarrow\mathfrak{O}$ of the form
\begin{equation*}
S_n^a:=U^a(f_n,c_n)
\end{equation*}
 for every $n<\omega$ and $a\in\{0,1\}$,
 where $f_n:X\rightarrow \mathbb{R}$ is a continuous function
 and $c_n$ is a real number for every $n<\omega$.

We say that $c$ is a \emph{local maximum (resp. local minimum)} of $f$
 if $c$ is a maximum (resp. minimum) value of $f|_V$ for some open subset $V$.
Local maxima and local minima are collectively called \emph{local extrema}.
If $c$ is a local maximum of $f:X\rightarrow \mathbb{R}$,
 then there exists a point $x\in U^{\bd}(f,c)$ with its neighbourhood $V$
 such that $V\cap U^1(f,c)$ is empty.
The point $x$ belongs to neither $\cl U^1(f,c)$ nor $U^0(f,c)$.
Therefore, $U^0(f,c)$ is not the exterior of $U^1(f,c)$.
Similarly, if $c$ a local minimum,
 then there exists $y\not\in \cl U^0(f,c) \cup U^1(f,c)$.
Hence, local extrema should be avoided
 in order to obtain a proper dyadic subbase.
We do not fix real numbers $c_n$ first,
 but give open intervals $I_n$ from which $c_n$ will be taken.
\begin{lem}\label{lem:countable-fc}
There exist a sequence $(f_n)_{n<\omega}$ of continuous functions
and a sequence $(I_n)_{n<\omega}$ of open intervals such that
the family $\{U^a(f_n,c_n) : n<\omega,\, a\in\{0,1\}\}$ is a subbase of $X$
if $c_n\in I_n$ for $n<\omega$.
\end{lem}
\begin{proof}
Since $X$ is separable, there exists a dense countable set
 $\{x_n\in X: n<\omega\}$.
Suppose that $\{B_n: n<\omega\}$ is a family consisting of open intervals
 that forms a base of the space $\mathbb{R}_{>0}$ of positive real numbers.
Note that if $b_n\in B_n$ for all $n<\omega$, then the set $\{b_n: n<\omega\}$
 is dense in $\mathbb{R}_{>0}$.
Let $n\mapsto (n_0,n_1)$ be a map from $\omega$ onto $\omega\times \omega$.
We define $f_n(x):=d(x_{n_0},x)$ and $I_n:=B_{n_1}$.
If $c_n\in I_n$ for all $n<\omega$, then the family $\{U^a(f_n,c_n): n<\omega,\, a\in\{0,1\} \}$ is a subbase of $X$
 because $\{U^0(f_n,c_n): n<\omega\}$ forms a base.
By definition, $U^0(f_n,c_n)$ and $U^1(f_n,c_n)$ are disjoint for all $n<\omega$.
\end{proof}

As the following lemma shows, the set of local extrema is countable.
\begin{lem}\label{lem:countable-extrema}
Every function $f:X\rightarrow\mathbb{R}$ has at most countably many local extrema.
\end{lem}
\begin{proof}
Let $\mathfrak{B}$ be a countable base of $X$.
Since each local extremum is an extremum of $f|_{B}$ for some $B\in\mathfrak{B}$,
 the number of local extrema of $f$ is countable.
\end{proof}

In order to obtain the properness property,
we have only to avoid local extrema of finitely many functions.
Therefore, we can avoid them.
\begin{lem}\label{lem:non-extrema}
Let $A$ be a subset of $X$,
 $f:X\rightarrow\mathbb{R}$ a continuous function,
 $c$ a real number.
If $c\in \mathbb{R}$ is not a local extremum of $f|_{\cl A}$,
 then we have $\cl A\setminus U^{1-a}(f,c) =\cl(A\cap U^a(f,c))$ for $a\in \{0,1\}$.
\end{lem}
\begin{proof}
We can see
 $\cl A\setminus  U^{1-a}(f,c)
 \supseteq \cl A\cap \cl U^a(f,c)
 \supseteq \cl(A\cap U^a(f,c))$.
Suppose that $x\in \cl A\setminus  U^{1-a}(f,c)$
 and $V$ is an open neighbourhood of $x$.
Since $c$ is not a local extremum of $f|_{\cl A}$,
 there exists $y\in V\cap \cl A$ such that $y\in U^a(f,c)$,
 i.e., $V\cap \cl A\cap U^a(f,c) \neq \emptyset$.
Since $U^a(f,c)$ and $V$ are open, we have $V\cap A \cap U^a(f,c) \neq \emptyset$.
Since this holds for every open neighbourhood $V$ of $x$,
 we obtain $x\in \cl(A\cap U^a(f,c))$.
\end{proof}

\begin{proof}[Proof of Proposition \ref{prop:existence-proper}]
By Lemma \ref{lem:countable-fc}
 we can take a sequence $(f_n)_{n<\omega}$ of continuous functions
 and a sequence $(I_n)_{n<\omega}$ of open intervals, such that
 the family $\{U^a(f_n,c_n) : n<\omega,\, a\in\{0,1\}\}$ is a subbase of $X$
 if $c_n\in I_n$ for $n<\omega$.

First, we take $c_0\in I_0$ which is not a local extremum of $f_0$,
 and set $S_0^a:=U^a(f_0,c_0)$ for $a\in\{0,1\}$.
By Lemma \ref{lem:non-extrema},
 $S_0^0$ and $S_0^1$ are the exteriors of each other.

Let $n$ be a finite ordinal.
Suppose that we have obtained a family $\{S_k^a : k<n, a\in \{0,1\}\}$
 such that $\bar{S}(\sigma)=\cl S(\sigma)$ for all $\sigma \in \mathbb{T}^n$.
We take a real number $c_n\in I_n$ which is not a local extremum of
 $f_n|_{\bar{S}(\sigma)}$ for all $\sigma \in \mathbb{T}^n$.
We set $S_n^a:=U^a(f_n,c_n) $ for $a\in \{0,1\}$.
For all $\sigma\in \mathbb{T}^n$ and $a\in \{0,1\}$,
 we have
 \begin{equation*}
\bar{S}(\sigma[n\mapsto a])=\bar{S}(\sigma)\setminus S_n^{1-a}
 =\cl S(\sigma)\setminus S_n^{1-a}
\end{equation*}
 by the assumption.
Since $c_n$ is not a local extremum of $f_n|_{\cl S(\sigma)}$,
 by Lemma \ref{lem:non-extrema}, we obtain
\begin{equation*}
\cl S(\sigma) \setminus S_n^{1-a}= \cl (S(\sigma)\cap S_n^a)
= \cl S(\sigma[n\mapsto a]).
\end{equation*}
Therefore,
 $\bar{S}(\sigma)=\cl S(\sigma)$ holds for all $\sigma \in \mathbb{T}^{n+1}$.
Hence, we obtain a proper dyadic subbase inductively.
\end{proof}

\subsection{Fat points}
In the proof of Proposition \ref{prop:existence-proper},
 $c_n$ could be a local extremum of $f_n|_{f_m^{-1}(c_m)}$ for some $m,n<\omega$.
If $c_n$ is a local extremum of $f_n|_{f_m^{-1}(c_m)}$,
then $U^0(f_n,c_n)\cap U^{\bd}(f_m,c_m)$ and $U^1(f_n,c_n)\cap U^{\bd}(f_m,c_m)$
 are not the exteriors of each other in the space $U^{\bd}(f_m,c_m)$,
 and we fail to obtain the strong properness.
We have to avoid $c_m$ such that
 an already chosen $c_n$ will be a local extremum of $f_n|_{f_m^{-1}(c_m)}$.
In Section \ref{sec:proof-st-proper},
 we show that such real numbers can be avoided if the space $X$ is locally compact.

In the rest of this section, $X$ is a locally compact separable metric space.
\begin{defi}
A subset $A\subseteq X$ is
\begin{enumerate}
\renewcommand{\labelenumi}{(\roman{enumi})}
\item \emph{codense} if $\inte A=\emptyset$.
\item \emph{nowhere dense} if $\inte \cl A =\emptyset$.
\item \emph{meagre} if $A$ is a countable union of nowhere dense subsets.
\end{enumerate}
\end{defi}

Let $r$ be a non-negative integer, $f:X\rightarrow \mathbb{R}^r$ a continuous map.
We say that a point $x\in X$ is \emph{fat} with respect to $f$
 if $f(V)$ has an interior point for every neighbourhood $V$ of $x$.
The set of all fat points with respect to $f$ is denoted by $\fat_f X$.
For every subset $A\subseteq X$,
 $\fat_f A$ denotes the set of fat points of $A$ with respect to $f|_A$.
\begin{lem}\label{lem:fat-inclusion}
Let $f:X\rightarrow \mathbb{R}^r$ be a continuous map.
For every subset $A\subseteq X$,
 we have $\fat_f A\subseteq \fat_f X\cap A$.
\end{lem}
\begin{proof}
Let $x\in \fat_f A$ be a point, $V\subseteq X$ its open neighbourhood.
We can see $x\in A$.
Since $x$ belongs to $\fat_f A$, $f(A\cap V)$ has an interior point.
Therefore, $f(V)\supseteq f(A\cap V)$ also has an interior point,
 and hence $x\in \fat_f X$.
\end{proof}
By definition, $X$ has no fat point with respect to $f$ if $f(X)$ is codense.
We show its converse.
\begin{prop}\label{prop:interior-fat-existence}
Let $r$ be a non-negative integer, $f:X\rightarrow \mathbb{R}^r$ a continuous map.
If $\fat_f X$ is empty, then $f(X)$ is codense. 
\end{prop}
We make a remark about the case in which $r$ is zero.
$\mathbb{R}^0$ is a one point set
 and every point $x\in X$ is mapped to the same point by $f$.
Therefore, we have $\fat_f X=X$.
Note that for every subset $A$ of a one point set, we have
\begin{equation*}
A\text{ is codense }\Leftrightarrow A\text{ is nowhere dense } \Leftrightarrow
 A\text{ is meagre }
\Leftrightarrow A=\emptyset.
\end{equation*}
Therefore, Proposition \ref{prop:interior-fat-existence} holds in this case.

Baire category theorem states that
 every meagre subset of a complete metric space is codense.
Since $\mathbb{R}^r$ is a complete metric space, we have the following.
\begin{lem}[Baire category theorem]\label{BCT}
Every meagre subset of $\mathbb{R}^r$ is codense.\qed
\end{lem}
\begin{proof}[Proof of Proposition \ref{prop:interior-fat-existence}]
Suppose that $X$ has no fat point and $\{B_n: n<\omega \}$ is a base of $X$.
Since $X$ is locally compact,
 all $B_n$ can be chosen to be relatively compact.
For all $x\in X$, there exists a compact neighbourhood $\cl B_n$ of $x$
 such that $f(\cl B_n)$ is codense.
Note that the image of a compact set by a continuous map is always compact
 and every compact codense subset is nowhere dense.
Therefore, we have 
 $X=\bigcup\{ \cl B_n : f(\cl B_n) \text{ is nowhere dense}\}.$
We can see that
 $f(X)=\bigcup \{ f(\cl B_n): f(\cl B_n) \text{ is nowhere dense}\}$
 is meagre,
and hence $f(X)$ is codense by Lemma \ref{BCT}.
\end{proof}

Let $f,g:X\rightarrow\mathbb{R}$ be continuous functions,
 $c$ a real number.
We consider whether
 there exists $b\in \mathbb{R}$ such that
 $b$ is a local extremum of neither $g$ nor $g|_{f^{-1}(c)}$
 and $c$ is not a local extremum of $f|_{g^{-1}(b)}$.
As the next proposition shows,
 if $c$ is not a local extremum of $f|_{\fat_g X}$,
 then the set
 $\{p\in \mathbb{R} : c\text{ is a local extremum of }f|_{g^{-1}(p)}\}$
 is meagre, and such a $b$ exists.
Since we will deal with countably many functions,
 we prove this in an extended form.
\begin{prop}\label{prop:high-dim}
Let $f:X\rightarrow\mathbb{R}$ be a continuous function,
$r$ a non-negative integer,
$g:X\rightarrow \mathbb{R}^r$ a continuous map,
$c$ a real number which is not a local extremum of $f|_{\fat_g X}$.
The set $\{p\in \mathbb{R}^r : c\text{ is a local extremum of }f|_{g^{-1}(p)}\}$
 is meagre. 
\end{prop}
Before proving Proposition \ref{prop:high-dim}, we provide two lemmas.
\begin{lem}\label{lem:compact-cut-nowhere-dense}
Suppose that $f,g$ and $c$ are as above.
For any open subset $B$ of $X$ and any compact subset $K$ of $B$,
 $g(K\cap U^{\bd}(f,c)) \setminus g(B\cap U^a(f,c))$
 is nowhere dense for $a\in \{0,1\}$.
\end{lem}
\begin{proof}
For each $a\in\{0,1\}$, we have 
\begin{equation*}
g(K\cap U^{\bd}(f,c)) \setminus g(B\cap U^a(f,c))
\subseteq g(K\cap U^{\bd}(f,c)) \setminus \inte g(B\cap U^a(f,c)).
\end{equation*}
The right hand side is closed, and we will show that it is codense.
Let $V\subseteq g(K\cap U^{\bd}(f,c))$ be a non-empty open subset.
We can see that $K\cap U^{\bd}(f,c) \cap g^{-1}(V)$
 is a locally compact separable metric space.
Therefore, by Proposition \ref{prop:interior-fat-existence},
 the set $\fat_g (K\cap U^{\bd}(f,c) \cap g^{-1}(V))$ is not empty.
By Lemma \ref{lem:fat-inclusion}, we obtain 
\begin{align*}
 \emptyset \neq \fat_g (K\cap U^{\bd}(f,c) \cap g^{-1}(V))
 &\subseteq \fat_g X \cap K \cap U^{\bd}(f,c) \cap g^{-1}(V)\\
 &\subseteq \fat_g X\cap B\cap U^{\bd}(f,c)\cap g^{-1}(V).
\end{align*}
Since $c$ is not a local extremum of $f|_{\fat_g X}$, 
 we obtain $\fat_g X\cap B\cap U^a(f,c)\cap g^{-1}(V)\neq \emptyset $
 for $a\in \{0,1\}$,
 and thus $g(B\cap U^a(f,c))$ has an interior point in $V$.
Thus, every non-empty open subset $V$ of $g(K\cap U^{\bd}(f,c))$
 intersects with $\inte g(B\cap U^a(f,c))$ non-trivially and,
 consequently, $g(K\cap U^{\bd}(f,c))\setminus \inte g(B\cap U^a(f,c))$
 is codense for $a\in \{0,1\}$.
\end{proof}
\begin{lem}\label{lem:on-each-open}
Suppose that $f,g$ and $c$ are as above.
For any open subset $B$ of $X$,
 $g(B\cap U^{\bd}(f,c)) \setminus g(B\cap U^a(f,c))$
 is meagre for $a\in \{0,1\}$.
\end{lem}
\begin{proof}
Since $B$ is a locally compact separable metric space,
 there are compact sets $K_n$, $n<\omega$, such that
 $B=\bigcup_{n<\omega} K_n$.
By Lemma \ref{lem:compact-cut-nowhere-dense},
 $g(K_n\cap U^{\bd}(f,c))\setminus g(B\cap U^a(f,c))$
 is nowhere dense for all $n<\omega$ and $a\in \{0,1\}$.
Therefore,  for each $a\in \{0,1\}$, their union
\begin{equation*}
\bigcup_{n<\omega} g(K_n\cap U^{\bd}(f,c))\setminus g(B\cap U^a(f,c))
=g(B\cap U^{\bd}(f,c))\setminus g(B\cap U^a(f,c))
\end{equation*}
 is meagre.
\end{proof}
\begin{proof}[Proof of Proposition \ref{prop:high-dim}]
Let $\{B_n: n<\omega\}$ be a countable base of $X$.
By Lemma \ref{lem:on-each-open},
 $g(B_n\cap U^{\bd}(f,c))\setminus g(B_n\cap U^a(f,c))$
is meagre for all $n<\omega$ and $a\in \{0,1\}$.
Therefore, their union
\begin{equation*}
M:=\widebigcup{n<\omega\\ a\in\{0,1\}} g(B_n\cap U^{\bd}(f,c))\setminus g(B_n\cap U^a(f,c))
\end{equation*}
is meagre.

Suppose that $c$ is a local extremum of $f|_{g^{-1}(p)}$ for a point $p\in\mathbb{R}^r$.
There exists a base element $B_n$
 such that $B_n\cap g^{-1}(p)\cap U^{\bd}(f,c)$ is non-empty,
 and either $B_n\cap g^{-1}(p) \cap U^0(f,c)$ or $B_n\cap g^{-1}(p) \cap U^1(f,c)$ is empty.
That is, $p$ belongs to $g(B_n\cap U^{\bd}(f,c))\setminus g(B_n\cap U^a(f,c))$ for an $a\in \{0,1\}$.
Therefore, we obtain
\begin{equation}
\{p\in \mathbb{R}^r : c\text{ is a local extremum of }f|_{g^{-1}(p)}\}\subseteq M.
\tag*{\qEd}
\end{equation}
\def\popQED{}
\end{proof}

We show that if the set of $(r+1)$-tuples that we should avoid is meagre,
 then we have only to avoid a meagre subset at each step.
For a real number $c$, a hyperplane $H_c\subseteq \mathbb{R}^{r+1}$
 is defined as $\{(x_0,\dots,x_r): x_r=c\}$.
\begin{prop}\label{prop:meagre-fibre-meagre}
If $M\subseteq \mathbb{R}^{r+1}$ is meagre,
then the set $\{c\in \mathbb{R} : M\cap H_c\text{ is meagre in }H_c\}$ is comeagre,
 i.e., its complement is meagre.
\end{prop}
For the proof we need a lemma.
\begin{lem}\label{lem:fibre-codense-comeagre}
If $K\subseteq \mathbb{R}^{r+1}$ is closed nowhere dense,
then $\{ c\in \mathbb{R}: K\cap H_c$ is nowhere dense in $ H_c\}$ is comeagre.
\end{lem}
\begin{proof}
Let $\{B_n: n<\omega\}$ be a countable base of $\mathbb{R}^r$.
For a real number $c$, $B_n\times\{c\}$
 denotes the set $\{(x_0,\dots,x_{r-1},c): (x_0,\dots,x_{r-1})\in B_n\}$.
Note that if $K\cap H_c$ is not codense in $H_c$,
 then an open subset of $H_c$ is contained in $K$,
 and therefore, $B_n\times \{c\}\subseteq K$ for some $n<\omega$.
For each $n<\omega$,
 the set $\{ c\in \mathbb{R}: B_n\times \{c\} \subseteq K\}$ is nowhere dense
 because $K$ is nowhere dense.
We can see that their union 
\begin{equation*}
\bigcup_{n<\omega}\{ c\in \mathbb{R}: B_n\times \{c\} \subseteq K\}=
\{ c\in \mathbb{R}: K\cap H_c \text{ is not codense in } H_c\}
\end{equation*}
is meagre.
Therefore, its complement
 $\{ c\in \mathbb{R}: K\cap H_c \text{ is codense in } H_c\}$ is comeagre.
Since $K$ is closed, $K\cap H_c$ is closed in $H_c$.
Hence, 
 $\{ c\in \mathbb{R}: K\cap H_c \text{ is nowhere dense in } H_c\}$ is comeagre.
\end{proof}
\begin{proof}[Proof of Proposition \ref{prop:meagre-fibre-meagre}]
Suppose that $M\subseteq \mathbb{R}^{r+1}$ is meagre.
There exists a countable covering $M\subseteq \bigcup_{n<\omega} K_n$,
 where $K_n$ is closed nowhere dense for $n<\omega$.
By Lemma \ref{lem:fibre-codense-comeagre},
 $\{c\in \mathbb{R} : K_n\cap H_c \text{ is nowhere dense in }H_c\}$
 is comeagre for all $n<\omega$.
By De Morgan's law, the intersection of countably many comeagre subsets is comeagre.
Therefore, their intersection
\begin{equation*}
C:=\bigcap_{n<\omega}\{c\in \mathbb{R} : K_n\cap H_c \text{ is nowhere dense in }H_c\}\\
\end{equation*}
 is comeagre.
If $c\in C$, then $K_n\cap H_c$ is nowhere dense in $H_c$ for all $n<\omega$,
 and therefore,
 $\bigcup_{n<\omega}K_n\cap H_c\supseteq M\cap H_c$ is meagre in $H_c$.
 Hence, the set
 $\{c\in \mathbb{R} : M\cap H_c \text{ is meagre in }H_c\}\supseteq C$
 is comeagre.
\end{proof}

\subsection{Proof of Theorem \ref{thm:locally-compact-st-proper}}
\label{sec:proof-st-proper}

\subsubsection{Construction of a strongly proper dyadic subbase}
By Lemma \ref{lem:countable-fc},
 we can take a sequence $(f_n)_{n<\omega}$ of real-valued continuous functions
 and a sequence $(I_n)_{n<\omega}$ of open intervals in $\mathbb{R}$, such that
 the family $\{U^a(f_n,c_n) : n<\omega,\, a\in\{0,1\}\}$ is a subbase of $X$
 if $c_n\in I_n$ for $n<\omega$.

Every subset of $\omega$ is represented by
 the domain of a sequence of $\{\top,\bot\}$,
 where the \emph{top character} $\top$ means non-bottom.
For a sequence $\upsilon\in\{\top,\bot\}^*$,
$f_{\upsilon}$ denotes the map $f_{\upsilon}:X\rightarrow \mathbb{R}^{\dom(\upsilon)}$
given by $f_{\upsilon}(x)=(f_k(x))_{k\in \dom(\upsilon)}$.

Let $n>0$ be a finite ordinal.
If we have a sequence $(c_i)_{i<n}$
 and set $S_i^a=U^a(f_i,c_i)$ for all $i<n$ and $a\in \{0,1,\bd\}$,
 then we define $M(k,\tau,\upsilon)\subseteq \mathbb{R}^{\dom(\upsilon)}$ as
\begin{equation*}
M(k,\tau,\upsilon):=\{p\in \mathbb{R}^{\dom(\upsilon)} :
 c_k\text{ is a local extremum of }f_k|_{S(\tau)\cap f_{\upsilon}^{-1}(p)}\}
\end{equation*}
for all $\tau\in \{\bd,\bot\}^n$,
 $\upsilon =\bot^n\upsilon'\in\{\top,\bot\}^*$ and $k\in n\setminus \dom(\tau)$.
For each $n$, we now define inductively a sequence $(c_i)_{i<n}$ such that
\begin{equation}
 \forall \tau\in\{\bd,\bot\}^n.\ 
 \forall \upsilon=\bot^n\upsilon' \in \{\top,\bot\}^*.\ 
 \forall k \in n\setminus \dom(\tau).\ 
M(k,\tau,\upsilon)
\text{ is meagre.}
\label{eq:st-proper-condition}
\end{equation}
%

First we define $c_0\in \mathbb{R}$ as follows.
By Lemma \ref{lem:countable-extrema},
 the set of local extrema of $f_0|_{\fat_{f_\upsilon}X}$ is countable
 for all $\upsilon\in \{\top,\bot\}^*$ with $\upsilon(0)=\bot$.
Their union $E_0$ is countable.
Therefore, we can take $c_0\in I_0\setminus E_0$.
We set $S_0^a:=U^a(f_0,c_0)$ for $a\in \{0,1,\bd\}$.
Since $c_0\not\in E_0$,
 by Proposition \ref{prop:high-dim},
the set $M(0,\bot,\upsilon)$
is meagre for all $\upsilon=\bot\upsilon'\in \{\top,\bot\}^*$.
Therefore, \eqref{eq:st-proper-condition} holds for $n=1$.

Let $n>0$ be a finite ordinal.
Suppose that 
we have obtained a sequence $(c_i)_{i<n}$ such that \eqref{eq:st-proper-condition} holds.
For $c\in \mathbb{R}$ and $\upsilon\in \{\top,\bot\}^*$ with $n\in \dom(\upsilon)$,
we define a hyperplane
$H_c(n,\upsilon):=\{(x_k)_{k\in \dom(\upsilon)}\in \mathbb{R}^{\dom(\upsilon)}: x_n=c\}$.
By Proposition \ref{prop:meagre-fibre-meagre},
 if $n\in \dom(\upsilon)$, then the set
\begin{equation*}
\{c \in\mathbb{R} : 
M(k,\tau,\upsilon)\cap H_c(n,\upsilon)
\text{ is meagre in } H_c(n,\upsilon)\}
\end{equation*}
 is comeagre for all $\tau\in\{\bd,\bot\}^n$, $\upsilon=\bot^n\upsilon' \in \{\top,\bot\}^*$ and
 $k \in n\setminus \dom(\tau)$.
Therefore, their intersection $C_n$ is comeagre. 
By Lemma \ref{lem:countable-extrema},
 the set of local extrema of $f_n|_{\fat_{f_\upsilon}S(\tau)}$ is countable
 for all $\tau\in \{\bd,\bot\}^n$ and $\upsilon=\bot^{n+1}\upsilon'\in \{\top,\bot\}^*$.
Their union $E_n$ is countable.
We can take $c_n$ from $(I_n\cap C_n) \setminus E_n$,
 and we set $S_n^a:=U^a(f_n,c_n)$ for $a\in\{0,1,\bd\}$.
Since $c_n\in C_n$, we obtain 
\begin{multline*}
\forall \tau\in \{\bd,\bot\}^n.\ \forall \upsilon=\bot^n\top \upsilon'\in \{\top,\bot\}^*.\ 
\forall k\in n\setminus\dom(\tau). \\
M(k,\tau[n\mapsto\bd],\upsilon[n\mapsto \bot])\text{ is meagre.}
\end{multline*} 
Note that $S(\tau)$ is a locally compact separable metric space for all $\tau\in\{\bd,\bot\}^n$. 
Since $c_n\not\in E_n$, by Proposition \ref{prop:high-dim},
$M(n,\tau,\upsilon)$ is meagre for all $\tau\in \{\bd,\bot\}^n$ and $\upsilon=\bot^{n+1}\upsilon'\in \{\top,\bot\}^*$.
Therefore, we can obtain a sequence $(c_n)_{n<\omega}$
which satisfies \eqref{eq:st-proper-condition} for all $n<\omega$ inductively.

\subsubsection{Proof of strong properness}
Suppose that $(f_n)_{n<\omega}$ is a sequence of functions,
 $(c_n)_{n<\omega}$ is a sequence of real numbers,
 \eqref{eq:st-proper-condition} holds for all $n<\omega$ and 
 $\{S_n^a : n<\omega,\, a\in \{0,1\}\}$ forms a subbase of $X$,
 where $S_n^a=U^a(f_n,c_n)$ for $n<\omega$ and $a\in \{0,1\}$.
We can easily see that $S$ is a dyadic subbase of $X$.
\begin{lem}\label{lem:extrema-on-bd}
For all $n<\omega$, $c_n$ is not a local extremum of $f_n|_{\bar{S}(\sigma)}$ for all $\sigma\in \{0,1,\bd,\bot\}^*$
 with $\sigma(n)=\bot$.
\end{lem}
\begin{proof}
In \eqref{eq:st-proper-condition}, setting $\upsilon:=\bot$, 
 we can see that $c_k$ is not a local extremum of $f_k|_{S(\tau)}$ for all $\tau\in\{\bd,\bot\}^n$
 and $k\in n\setminus \dom(\tau)$.
Since \eqref{eq:st-proper-condition} holds for all $n<\omega$,
 $c_n$ is not a local extremum of $f_n|_{S(\tau)}$ for all $\tau\in\{\bd,\bot\}^*$ and $n\in\omega\setminus\dom(\tau)$.

Suppose that $c_n$ is a local extremum of $f_n|_{\bar{S}(\sigma)}$ for a sequence $\sigma\in \{0,1,\bd,\bot\}^*$
with $\sigma(n)=\bot$.
There exists an open subset $V\subseteq X$
 such that $c_n$ is an extremum of $f_n|_{V\cap \bar{S}(\sigma)}$.
Take a point $x\in V\cap \bar{S}(\sigma)\cap f_n^{-1}(c_n)$
 and set
\begin{equation*}
 W:=V\cap \widebigcap{k\,\in\,\dom(\sigma) \\ \cap\dom(\varphi_S(x))}\, S_k^{\varphi_S(x)(k)}.
\end{equation*}
Let $\tau\in \{\bd,\bot\}^*$ be a sequence
 whose domain is the set $\dom(\sigma)\setminus \dom(\varphi_S(x))$.
For all $k\in \dom(\sigma)$,
 either $S_k^{\varphi_S(x)(k)}$ or $S_k^{\tau(k)}$ is defined,
 and the defined one is contained in $S_k^{\sigma(k)}\cup S_k^{\bd}$.
Therefore,
 we obtain
\begin{align*}
W\cap S(\tau)=V\cap \widebigcap{k\,\in\,\dom(\sigma) \\\cap \dom(\varphi_S(x))} S_k^{\varphi_S(x)(k)} \cap
\widebigcap{\scriptsize \begin{aligned} k\in&\dom(\sigma)\\[-1pt]&\mathrlap{\setminus\dom(\varphi_S(x))}\end{aligned}} S_k^{\bd} 
 \quad\subseteq\quad V\cap \widebigcap{k\,\in\,\dom(\sigma)} (S_k^{\sigma(k)}\cup S_k^{\bd})
 = V\cap \bar{S}(\sigma).
\end{align*}
We can see that $c_n$ is an extremum value of $f|_{W\cap S(\tau)}$
 because $W\cap S(\tau)$ is a subset of $V\cap \bar{S}(\sigma)$ and contains the point $x\in f_n^{-1}(c_n)$.
Since $W$ is open, $c_n$ is a local extremum of $f|_{S(\tau)}$, a contradiction.
\end{proof}
\begin{prop}
 $S$ is a strongly proper dyadic subbase.
\end{prop}
\begin{proof}
Let $\sigma\in \{0,1,\bd,\bot\}^*$ be a sequence,
\begin{equation*}
 \sigma_{\bd}(n):=\left\{
 \begin{array}{ll}
  \bd &(\text{if }\sigma(n)=\bd)\\
  \bot & (\text{otherwise})
 \end{array}\right.
\end{equation*}
 its restriction.
Since the cardinality of $\sigma^{-1}(\{0,1\})$ is finite,
 we can take a finite sequence $(\tau_k)_{k<m}$
 such that $\tau_0=\sigma_{\bd}$,
 $\tau_{k+1}=\tau_k[n\mapsto \sigma(n)]$ for some $n\in \dom(\sigma)\setminus \dom(\tau_k)$ for all $k<m$,
 and $\tau_{m-1}=\sigma$.
We show that $\bar{S}(\tau_k)=\cl S(\tau_k)$ holds for all $k<m$ by induction.
By definition, we have $\bar{S}(\tau_0)=S(\tau_0)=\cl S(\tau_0)$.
Assume that we have $\bar{S}(\tau_k)=\cl S(\tau_k)$ and $\tau_{k+1}=\tau_k[n\mapsto a]$ for some $k<m$,
 $n\in \dom(\sigma)\setminus \dom(\tau_k)$ and $a=\sigma(n)\in \{0,1\}$.
By the assumption, we have $\bar{S}(\tau_k[n\mapsto a])=\bar{S}(\tau_k)\cap \cl S_n^a=\cl S (\tau_k)\cap \cl S_n^a$.
By Lemma \ref{lem:extrema-on-bd}, $c_n$ is not a local extremum of $f_n|_{\bar{S}(\tau_k)}$.
By Lemma \ref{lem:non-extrema},
we obtain $\cl S(\tau_k) \cap \cl S_n^a=\cl (S(\tau_k)\cap S_n^a)=\cl S(\tau_{k+1})$. 
\end{proof}

\section*{Acknowledgement}
I would like to thank Hideki Tsuiki for his expert guidance and valuable discussions.
I am also grateful to Haruto Ohta for technical advices.
Finally, I thank the referees for valuable comments.


\end{document}